\documentclass[12pt]{amsart}
\usepackage{amssymb,amsmath}
\usepackage[utf8]{inputenc}
\usepackage{xcolor}
\usepackage[margin=2.6cm]{geometry}
\usepackage{hyperref}
\usepackage[capitalise]{cleveref}

\hypersetup{
	colorlinks=true,
	linkcolor=blue,
	citecolor=blue,
	urlcolor=blue
}

\newtheorem{proposition}{Proposition}[section]
\newtheorem{lemma}[proposition]{Lemma}
\newtheorem{corollary}[proposition]{Corollary}
\newtheorem{theorem}[proposition]{Theorem}

\newtheorem{example}[proposition]{Example}

\DeclareMathOperator{\R}{\mathbb {R}}
\DeclareMathOperator{\depth}{depth}
\DeclareMathOperator{\pd}{pd}
\DeclareMathOperator{\reg}{reg}
\DeclareMathOperator{\Min}{Min}
\DeclareMathOperator{\lk}{lk}

\DeclareMathOperator{\Spec}{Spec}
\newcommand{\cI}{I_c}

\def\alb {\boldsymbol {\alpha}}
\def\m {\mathfrak m}

\begin{document}
	
	\pagenumbering{arabic}
	
\title[Regularity of symbolic powers of complementary  edges ideals]{Regularity of symbolic powers of \\ complementary  edge ideals}

    \author[Truong Thi Hien]{Truong Thi Hien}
	\address{Faculty of Natural Sciences, Hong Duc University, No. 565 Quang Trung, Hac Thanh, Thanh Hoa, Vietnam}
\email{hientruong86@gmail.com}

    \author[Manohar Kumar]{Manohar Kumar}
	\address{Department of Mathematics, Indian Institute of Technology Madras, Chennai 600036, India}
	\email{manhar349@gmail.com} 
	
	\dedicatory{Dedicated to Professor Le Tuan Hoa on the occasion of his 70th birthday}

	\thanks{Manohar Kumar is the corresponding author.}
	
	\subjclass[2020]{Primary 13D02, 13D45; Secondary 05E40, 05E99.}
	\keywords{Complementary edge ideal, symbolic power, Castelnuovo-Mumford regularity, Serre's condition, Cohen-Macaulay ring.}

\begin{abstract}
Let \(G\) be a finite simple graph on \(n\) vertices, and let \(I_c(G)\) be its complementary edge ideal. We determine \(\reg(I_c(G)^{(t)})\) for every \(t\geq1\) in terms of the number \(c(G)\) of nontrivial connected components of \(G\) and the presence of a \(K_2\)-component. Consequently, \(\reg(I_c(G)^{(t)})\leq \reg(I_c(G)^t) \)
for all \(t\geq1\), with equality for every \(t\) if and only if either \(c(G)=1\), or \(c(G)=2\) and \(G\) has a \(K_2\)-component. We also characterize Serre's condition \((S_2)\) for \(R/I_c(G)\), and classify the graphs for which \(R/I_c(G)^{(t)}\) is Cohen-Macaulay for every \(t\geq1\).
\end{abstract}

	\maketitle

\section{Introduction}

Square-free monomial ideals provide a natural link between commutative algebra and combinatorics. A particularly natural class consists of those generated in degree \(n-2\), since they admit a simple description in terms of graphs.  For \(F\subseteq[n]\), write \(x_F=\prod_{i\in F}x_i\) with \(x_\emptyset=1\). If \(G\) is a finite simple graph on \([n]\), its complementary edge ideal is defined by
\[
I_c(G)=\bigl(x_{[n]\setminus e}:e\in E(G)\bigr).
\]
Conversely, every square-free monomial ideal generated in degree \(n-2\) arises in this way. Complementary edge ideals were introduced independently by Hibi, Qureshi, and Saeedi Madani~\cite{HQSM26} and by Ficarra and Moradi~\cite{fm25}.

Symbolic powers of square-free monomial ideals reflect the structure of their minimal primes and have strong connections with combinatorics, convex geometry, and linear optimization; see, for instance, \cite{MontanoNunez21, HT}. While the asymptotic behavior of their Castelnuovo-Mumford regularity is known in considerable generality,
determining the exact regularity of individual symbolic powers is often more difficult.

Ficarra and Moradi~\cite{fm25} described the minimal primes of \(I_c(G)\), characterized several of its homological properties, and determined the regularity of all ordinary powers in terms of the connected components of \(G\). They also asked when \(I_c(G)\) satisfies Serre's condition \((S_2)\) \cite[Question~2.9]{fm25}, and whether
\[
\reg(I_c(G)^t)=\reg(I_c(G)^{(t)})
\qquad\text{for all }t\geq1
\]
\cite[Question~4.11]{fm25}. A related, but stronger, problem concerning the equality \(I_c(G)^t=I_c(G)^{(t)}\) was recently settled, in the equivalent language of \((n-2)\)-uniform clutters, by Roy and Saha~\cite{RS25}. Very recently, Ficarra, Moradi, and Muta studied the symbolic Rees algebra of complementary edge ideals and several homological invariants of their symbolic powers for special classes of graphs \cite{FicarraMoradiMuta}. They also announced related forthcoming work on the regularity and depth of symbolic powers. The results of the present paper were obtained independently.

The present paper gives complete answers to both questions of Ficarra and Moradi. Let \(c(G)\) denote the number of connected components of \(G\) containing at least two vertices. We determine \(\reg I_c(G)^{(t)}\) for every \(t\geq1\); see Theorem \ref{thm:main-symbolic-regularity}. In particular, for \(t\geq2\),
\[
\reg(I_c(G)^{(t)})
=
\begin{cases}
(n-2)t+1,
& \text{if \(c(G)\geq2\) and \(G\) has a \(K_2\)-component},\\
(n-2)t,
& \text{otherwise}.
\end{cases}
\]
Consequently,
\[
\reg(I_c(G)^{(t)}) \leq \reg(I_c(G)^t)
\qquad\text{for all }t\geq1,
\]
and equality holds for every \(t\geq1\) precisely when either \(c(G)=1\), or \(c(G)=2\) and \(G\) has a \(K_2\)-component; see Corollary \ref{cor:symbolic-ordinary-comparison}.

We also characterize Serre's condition and the Cohen-Macaulayness of symbolic powers. We prove that \(R/I_c(G)\) satisfies \((S_2)\) if and only if either \(G\) has no isolated vertices and is complete or has girth at least \(5\), or \(G\cong K_2\sqcup sK_1\) for some \(s\geq1\); see Theorem \ref{thm:S2}. Moreover, the Cohen-Macaulayness of a symbolic power $I_c(G)^{(t)}$ with \(t \geq 3\) already forces every symbolic power to be Cohen-Macaulay. This occurs precisely for complete graphs, stars, disjoint unions of edges, and \(K_2\) together with isolated vertices; see Theorem~\ref{thm:symbolic-CM}.

Our proofs combine symbolic polyhedra, degree complexes, Takayama's formula, and monomial localization. The symbolic polyhedron determines the slope of the regularity function, while degree complexes and localization allow us to recover its exact value. The results on Serre's condition and Cohen-Macaulayness are obtained from links in the associated Stanley-Reisner complex and from the matroid characterization of Cohen-Macaulay symbolic powers. Section~\ref{sec:preliminaries} collects the necessary preliminaries, Section~\ref{sec:symbolic-regularity} determines the regularity of the symbolic powers, and Section~\ref{sec:depth-properties} treats Serre's condition and Cohen-Macaulayness.

\section{Preliminaries}\label{sec:preliminaries}

In this section, we recollect notation, terminology and basic results used in the paper.  Throughout the paper, let $k$ be a field, let $R = k[x_1,\ldots,x_n],\  n\geqslant 1$ be a polynomial ring, and let $\m = (x_1,\ldots,x_n)$ be the maximal homogeneous ideal of $R$.

\subsection{Complementary edge ideals and symbolic powers}

Let \(G\) be a finite simple graph on \([n]=\{1,\ldots,n\}\). For \(F\subseteq[n]\), let \(P_F=(x_i:i\in F).\) For \(U\subseteq[n]\), let \(G_U\) denote the subgraph of \(G\) induced by \(U\). The complement of \(G\), denoted by \(G^c\), is the graph on \([n]\) with 
\[ E(G^c)=\binom{[n]}2\setminus E(G) =\bigl\{\{i,j\}:1\leq i<j\leq n,\ \{i,j\}\notin E(G)\bigr\}.\]

A \(3\)-clique of \(G\) is called a \emph{triangle}, and the set of all triangles of \(G\) is denoted by \(\mathcal T(G)\). The number of connected components of \(G\) having at least two vertices is denoted by \(c(G)\); thus isolated vertices are not counted. A connected component isomorphic to \(K_2\) is called a \(K_2\)-component.

If \(G\) contains a cycle, then its \emph{girth}, denoted by \(\operatorname{girth}(G)\), is the length of its shortest cycle. If \(G\) is a forest, then we set \(\operatorname{girth}(G)=\infty\). The complementary edge ideal of \(G\), denoted by \(I_c(G)\), which is defined as
\[
I_c(G) =\left(x_{[n]\setminus e}:e\in E(G)\right)\subseteq R.
\]
In particular, if \(E(G)\ne\emptyset\), then \(I_c(G)\) is generated in degree \(n-2\). The minimal primary decomposition of a complementary edge ideal is described as in the following lemma.

\begin{lemma}[{\cite[Theorem~2.1]{fm25}}]\label{lm:primary-decomposition}
Let \(G\) be a finite simple graph without isolated vertices. Then
\[
I_c(G)
=
\left(\bigcap_{e\in E(G^c)}P_e\right)
\cap
\left(\bigcap_{T\in\mathcal T(G)}P_T\right),
\]
and this decomposition is irredundant.
\end{lemma}

Throughout the paper, whenever symbolic powers of \(I_c(G)\) are considered, we assume that
\(I_c(G)\) is a nonzero proper ideal. In particular, \(n\geq3\) and \(E(G)\neq\emptyset\).

We recall the definition for symbolic powers that will be used throughout the paper. Let \(I\subsetneq R\) be a nonzero homogeneous ideal. For \(t\ge1\), the \(t\)-th symbolic power of \(I\) is defined by
\[
I^{(t)}
=
\bigcap_{P\in\Min(I)}
\bigl(I^tR_P\cap R\bigr).
\]
In particular, if \(I\) is a square-free monomial ideal, then
\[
I=\bigcap_{P\in\Min(I)}P
\qquad\text{and}\qquad
I^{(t)}=\bigcap_{P\in\Min(I)}P^t.
\]
Since \(I_c(G)\) is square-free, by  Lemma~\ref{lm:primary-decomposition}, for every \(t\ge1\),
\[
I_c(G)^{(t)}
=
\left(\bigcap_{e\in E(G^c)}P_e^t\right)
\cap
\left(\bigcap_{T\in\mathcal T(G)}P_T^t\right)
\]
whenever \(G\) has no isolated vertices.

\subsection{Regularity, depth, and simplicial complexes}

Let \(M\) be a nonzero finitely generated graded \(R\)-module. For \(i\ge0\), consider
\[
a_i(M)=\max\{j:H^i_{\mathfrak m}(M)_j\ne0\},
\]
where \(\max\emptyset=-\infty\). The Castelnuovo-Mumford regularity and  the depth of \(M\) are defined by
\[
\reg(M)=\max_{i\ge0}\{a_i(M)+i\}, \text{ and }
\depth(M)
=
\min\{i\ge0:H^i_{\mathfrak m}(M)\ne0\}.
\]
We write \(\pd_R(M)\) for the projective dimension of \(M\) over \(R\). Since \(R\) is a polynomial ring of dimension \(n\), the Auslander--Buchsbaum formula is
\[
\pd_R(M)=n-\depth(M).
\]
For a nonzero proper homogeneous ideal \(I\subset R\), we have the relation
\[
\reg(I)=\reg(R/I)+1.
\]

A simplicial complex $\Delta$ on $[n]$ is a collection of subsets of $[n]$ such that, whenever $F\in\Delta$ and $G\subseteq F$, one has $G\in\Delta$. Its maximal faces (with respect to inclusion) are called \emph{facets}, and their set is denoted by \(\mathcal F(\Delta)\). For $F \in \Delta$, the dimension of $F$ is defined to be $\dim F = |F| - 1$. The empty set, $\emptyset$, is the unique face of dimension $-1$, as long as $\Delta$ is not the void complex $\{ \}$ consisting of no subsets of \([n]\). The complex \(\Delta\) is called \emph{pure} if all its facets have the same cardinality. The dimension of \(\Delta\) is $\dim \Delta = \max \{ \dim F \colon F \in \Delta \}$.

For \(\sigma\in\Delta\), the link of \(\sigma\) in \(\Delta\) is its subcomplex and is determined by
\[
\lk_\Delta(\sigma)
=
\{F\in\Delta:F\cap\sigma=\emptyset,\ F\cup\sigma\in\Delta\}.
\]
The 1-skeleton of $\Delta$ is the graph with vertex set $V(\Delta)$ and edge set consisting of the two-element faces of $\Delta$. We say that $\Delta$ is connected if its 1-skeleton is connected.

The Stanley--Reisner ideal and the Stanley--Reisner ring of \(\Delta\) are
\[
I_\Delta=(x_F:F\subseteq[n],\ F\notin\Delta),
\qquad
k[\Delta]=R/I_\Delta.
\]
Conversely, every square-free monomial ideal \(I\subseteq R\) is the Stanley--Reisner ideal of a unique simplicial complex, denoted by \(\Delta(I)\). The irredundant primary decomposition of \(I_\Delta\) is
\[
I_\Delta
=
\bigcap_{F\in\mathcal F(\Delta)}P_{[n]\setminus F},
\]
and hence
\[
I_\Delta^{(t)}
=
\bigcap_{F\in\mathcal F(\Delta)}
P_{[n]\setminus F}^{\,t}
\qquad \text{ for all } t\ge1 .
\]

The Alexander dual of \(\Delta\) is defined by
$\Delta^*
=
\{[n]\setminus F:F\notin\Delta\}.$ The following two formulas relate the homological invariants of Stanley-Reisner ideals to the topology of simplicial complexes.

\begin{lemma}[{\cite[Theorem~13.13]{MS}}]\label{Hochster-Reg}
For every simplicial complex \(\Delta\),
\[
\reg(I_\Delta)
=
1+\max\left\{
q:
\widetilde H_{q-1}\bigl(\lk_\Delta(\sigma);k\bigr)\ne0
\text{ for some }\sigma\in\Delta
\right\}.
\]
\end{lemma}

\begin{lemma}[Terai's formula;  {\cite{T1}}]\label{TeraiF}
For every simplicial complex \(\Delta\) on \([n]\),
\[
\reg(I_\Delta)=\pd_R(R/I_{\Delta^*}).
\]
\end{lemma}
% The Takayama's formula (see \cite{MT11,T}), which generalizes Hochster's
% formula, is stated as follows.

% \begin{theorem}[Takayama]
% Let $I$ be a monomial ideal in $R=k[x_1,\ldots,x_n]$. For each $i\in\mathbb{Z}$ and
%$\mathbf{a}\in\mathbb{Z}^n$, there is an isomorphism of $k$-vector spaces
%\[
%H^i_{\mathfrak m}(S/I)_{\mathbf a}
%\cong
%\widetilde{H}^{\,i-|G_{\mathbf a}|-1}
%\bigl(\Delta_{\mathbf a}(I);k\bigr),
%\]
%and $H^i_{\mathfrak m}(S/I)_{\mathbf a}=0$ if $G_{\mathbf a}\notin\Delta\bigl(\sqrt{I}\bigr).$
%\end{theorem}
We next recall the following depth criterion.

\begin{lemma}\label{lem:depth-skeleton}
Let \(\Sigma\) be a simplicial complex with vertex set \([m]\), where \(m\ge2\), let \(T=k[x_1,\ldots,x_m]\), and \(\mathfrak n=(x_1,\ldots,x_m)\). Then 
\(\depth k[\Sigma]\ge2 \) if and only if the \(1\)-skeleton of \(\Sigma\) is connected. Consequently, if the \(1\)-skeleton is connected, then \(\pd_T(T/I_\Sigma)\le m-2,\) whereas, if it is disconnected, then \(\depth(k[\Sigma])=1\) and \(\pd_T(T/I_\Sigma)=m-1.\)
\end{lemma}

\begin{proof}
By the multigraded form of Hochster's formula, we have
\[
\depth(k[\Sigma])
=
\min\left\{
|F|+i+1:
F\in\Sigma,\ i\ge-1,\ 
\widetilde H_i\bigl(\lk_\Sigma(F);k\bigr)\neq0
\right\}.
\]

% Since \(\Sigma\) has vertex set \([m]\), the value \(0\) cannot occur, and hence 
Obviously, \(\depth(k[\Sigma])\ge1\). Suppose that the \(1\)-skeleton of \(\Sigma\) is connected. Then $\widetilde H_0(\Sigma;k)=0.$ Moreover, since the \(1\)-skeleton is connected and \(m\ge2\), every vertex belongs to an edge. Thus, \(\lk_\Sigma(v)\) is nonempty for every vertex \(v\), and hence $\widetilde H_{-1}\bigl(\lk_\Sigma(v);k\bigr)=0.$ It follows that \(\depth (k[\Sigma])\ge2\).

Conversely, if the \(1\)-skeleton is disconnected, then \(\Sigma\) is disconnected and $\widetilde H_0(\Sigma;k)\ne0.$ Taking \(F=\emptyset\) and \(i=0\), by the depth formula, we get \(\depth(k[\Sigma]) \le1\). Thus, $\depth(k[\Sigma])=1.$

The statement about projective dimension now follow from the Auslander-Buchsbaum formula over \(T\).
\end{proof}

Recall that a Noetherian ring \(R\) is said to satisfy Serre's
condition \((S_2)\) if
\[
\depth(R_{\mathfrak p})
\ge
\min\{2,\dim(R_{\mathfrak p})\}
\]
for every \(\mathfrak p\in\Spec(R)\).

For Stanley--Reisner rings, this condition can be described in terms
of links.

\begin{lemma}[{\cite[Corollary~2.4]{PFTY}}]\label{lem:S2-links}
Let \(\Delta\) be a simplicial complex. Then \(k[\Delta]\) satisfies Serre's condition \((S_2)\) if and only if \(\Delta\) is pure and \(\lk_\Delta(F)\) is connected for every \(F\in\Delta\) such that \(\dim(\lk_\Delta(F)) \ge1.\)
\end{lemma}

A simplicial complex \(\Delta\) is called a \emph{matroid complex} if, whenever \(F,G\in\Delta\) and \(|F|<|G|\), there exists \(j\in G\setminus F\) such that \(F\cup\{j\}\in\Delta\). The facets of a matroid complex are called its \emph{bases}, and their set is denoted by \(\mathcal B(\Delta)\). All bases have the same cardinality, called the rank. The dual matroid \(\Delta^\perp\) is defined by
\[
\mathcal B(\Delta^\perp)
=
\{[n]\setminus B:B\in\mathcal B(\Delta)\}.
\]
% Note that this matroid dual is different from the Alexander dual \(\Delta^*\). 
A vertex contained in no basis is called a \emph{loop}. The uniform matroid \(U_{r,n}\) consists of all subsets of \([n]\) of cardinality at most \(r\).

\subsection{Degree complexes and monomial localization}

The multigraded local cohomology of a monomial quotient can be described by means of degree complexes. Let \(I\subseteq R\) be a monomial ideal and let \(\boldsymbol\alpha=(\alpha_1,\ldots,\alpha_n)\in\mathbb Z^n\). Set $G_{\boldsymbol\alpha} = \{i:\alpha_i<0\}.$ For \(F\subseteq[n]\setminus G_{\boldsymbol\alpha}\), set
\(R_{F\cup G_{\boldsymbol\alpha}} = R[x_i^{-1}:i\in F\cup G_{\boldsymbol\alpha}].\) The degree complex of \(I\) at \(\boldsymbol\alpha\) is defined by
\begin{equation}\label{degree-complex}
\Delta_{\boldsymbol\alpha}(I)
=
\left\{
F\subseteq[n]\setminus G_{\boldsymbol\alpha}:
x^{\boldsymbol\alpha}\notin
I R_{F\cup G_{\boldsymbol\alpha}}
\right\}.
\end{equation}

The connection with local cohomology is given by Takayama's formula.

\begin{lemma}[Takayama's formula {\cite[Theorem~2.2]{T}}]\label{TA}
For every \(i\ge0\),
\[
\dim_{k}(H^i_{\mathfrak m}(R/I)_{\boldsymbol\alpha})
=
\dim_{k}
\widetilde H_{i-|G_{\boldsymbol\alpha}|-1}
\bigl(\Delta_{\boldsymbol\alpha}(I);k\bigr).
\]
\end{lemma}

For symbolic powers of Stanley--Reisner ideals, the facets of the degree complex have the following description.

\begin{lemma}{\cite[Lemma~1.3]{MT}}\label{MTr} Let \(\Delta\) be a simplicial complex on \([n]\), \(\boldsymbol\alpha\in\mathbb N^n\),and \(t \ge 1\). Then
\[
\mathcal F\bigl(\Delta_{\boldsymbol\alpha}(I_\Delta^{(t)})\bigr)
=
\left\{
F\in\mathcal F(\Delta):
\sum_{i\notin F}\alpha_i\le t-1
\right\}.
\]
\end{lemma}

When \(\boldsymbol\alpha\in\mathbb Z^n\), note that if $G_\alpha\notin\Delta$, then $\Delta_\alpha(I_\Delta^{(t)})$ is the void complex; in the case $G_\alpha\in\Delta$, we have the following form.

\begin{lemma}{\cite[Lemma~1.3]{HT2}}\label{HoaTrLem}
Let \(\Delta\) be a simplicial complex on \([n]\), \(t \ge 1\) and let \(\boldsymbol\alpha\in\mathbb Z^n\) such that \(G_\alpha\in\Delta\) . Then
\[
\mathcal F\bigl(\Delta_{\boldsymbol\alpha}(I_\Delta^{(t)})\bigr)
=
\left\{
F\in\mathcal F\bigl(\lk_\Delta(G_{\boldsymbol\alpha})\bigr):
\sum_{i\notin F\cup G_{\boldsymbol\alpha}}\alpha_i\le t-1
\right\}.
\]
\end{lemma}

The following relation between links and monomial localization will also be useful.

\begin{lemma}\label{lem:localization}
Let \(\Delta\) be a simplicial complex on \([n]\), let \(\sigma\subseteq[n]\), and let
\[
R_\sigma=R[x_i^{-1}:i\in\sigma],
\qquad
R'=k[x_i:i\notin\sigma].
\]
Then, 
\[
I_\Delta R_\sigma\cap R'
=
\begin{cases}
I_{\lk_\Delta(\sigma)}, & \text{if }\sigma\in\Delta,\\
R', & \text{if }\sigma\notin\Delta.
\end{cases}
\]
\end{lemma}

\begin{proof}
Using the irredundant primary decomposition of \(I_\Delta\), we have
\[
I_\Delta
=
\bigcap_{F\in\mathcal F(\Delta)}P_{[n]\setminus F}.
\]
If \(\sigma\notin\Delta\), then no facet of \(\Delta\) contains \(\sigma\). Hence every component becomes the unit ideal after localization, and \(I_\Delta R_\sigma\cap R'=R'.
\)
\noindent
Suppose that \(\sigma\in\Delta\). Then,
\[
I_\Delta R_\sigma\cap R'
=
\bigcap_{\substack{F\in\mathcal F(\Delta)\\ \sigma\subseteq F}}
(x_i:i\in[n]\setminus F).
\]
Moreover, $\mathcal F\bigl(\lk_\Delta(\sigma)\bigr)
=
\left\{
F\setminus\sigma:
F\in\mathcal F(\Delta),\ \sigma\subseteq F
\right\},$
and $([n]\setminus\sigma)\setminus(F\setminus\sigma)
=
[n]\setminus F.$
Hence
\[
I_{\lk_\Delta(\sigma)}
=
\bigcap_{\substack{F\in\mathcal F(\Delta)\\ \sigma\subseteq F}}
P_{[n]\setminus F}.
\]
The lemma is proved. 

\end{proof}

\subsection{Symbolic polyhedra and general bounds}

In this section, we recall the bounds for regularity of symbolic powers that will be used in the next section. 

For a vector $\alb = (\alpha_1,\ldots,\alpha_n)\in\R^n$, we set $|\alb| : = \alpha_1+\cdots+\alpha_n$.
%and for a nonempty bounded closed subset $S$ of $\R^n$ we set 
%$$\delta(S) : = \max\{|\alb| \mid \alb \in S\}.$$

Let \(I\subseteq R\) be a nonzero proper square-free monomial ideal with \(\Min(I)=\{P_{F_1},\ldots,P_{F_s}\}.\) The symbolic polyhedron of \(I\) is defined by 
\[
\operatorname{SP}(I)
=
\left\{
\boldsymbol x\in\mathbb R_{\ge0}^n:
\sum_{i\in F_j}x_i\ge1
\text{ for }j=1,\ldots,s
\right\}.
\]
By \cite[Equation~(4)]{HT}, the slope of the regularity of symbolic powers is determined by
\[
\delta(I)
=
\lim_{t\to\infty}\frac{\reg I^{(t)}}t
=
\max\left\{
|\boldsymbol v|:
\boldsymbol v\text{ is a vertex of }\operatorname{SP}(I)
\right\}.
\]

If $\Delta$ is a simplicial complex on the vertex set $[n]$, and \(I=I_\Delta\), by \cite[Theorem~2.3]{HT}, we have the upper bound for $\reg(I^{(t)})$ as 
\begin{equation}\label{eq:general-symbolic-bound}
\reg(I^{(t)})
\le
\delta(I)(t-1)+b(I)
\qquad \text{ for all }  t \ge 1,
\end{equation}
where $b(I)
=
\max\left\{
\reg(I_\Gamma):
\Gamma\text{ is a subcomplex of }\Delta,\ 
\mathcal F(\Gamma)\subseteq\mathcal F(\Delta)
\right\}.$

For a homogeneous ideal \(I\), let \(d(I)\) denote the largest degree of a minimal homogeneous generator of \(I\). Recall that $ d(I) \le \reg(I). $

For a square-free monomial ideal, by \cite[Remark~2.8]{HT},  we have the lower bound
\begin{equation}\label{eq:general-symbolic-lower}
\reg(I^{(t)}) \ge d(I)t \qquad \text{ for all } t\ge1.
\end{equation}

%The following estimate for the multigraded components of local cohomology will also be used.

\begin{lemma}[{\cite[Theorem~2.2]{HT}}]\label{HT}
Let \(I\subseteq R\) be a square-free monomial ideal. For every \(i\ge0\) and \(t\ge1\),
\[
a_i(R/I^{(t)})
\le
\delta(I)(t-1).
\]
In particular, if $H^i_{\mathfrak m}(R/I^{(t)})_{\boldsymbol\alpha}\ne 0,$
then $ |\boldsymbol\alpha| \le a_i(R/I^{(t)}) \le \delta(I)(t-1).$
\end{lemma}
%%%%%%%%%%%%%%%%%%%%%%%%%%%%%%%%%%%%%5

\section{Regularity of symbolic powers of complementary edge ideals}\label{sec:symbolic-regularity}
In this section, we study the regularity of symbolic powers of complementary edge ideals. Let $G$ be a simple graph without isolated vertices. By Lemma~\ref{lm:primary-decomposition}, the irredundant primary decomposition of $I_c(G)$ is
\begin{equation}\label{eq:primary-section3}
I_c(G)=\left(\bigcap_{e\in E(G^c)}P_e\right)\cap\left(\bigcap_{T\in\mathcal T(G)}P_T\right).
\end{equation}
Therefore, if $\Delta=\Delta(I_c(G))$, then $\mathcal F(\Delta)=\{[n]\setminus e:e\in E(G^c)\}\cup\{[n]\setminus T:T\in\mathcal T(G)\}.$

%%%%%%%%%%%%%%%%%%%%%%%%%%%%%%%%%%%%%%%5%
\subsection{Bounds for the regularity of symbolic powers}
\begin{proposition}\label{prop:delta-no-isolates}
Let $G$ be a simple graph on $[n]$, where $n\ge3$. If $G$ has no isolated vertices, then
$\delta(I_c(G))=n-2.$
\end{proposition}

\begin{proof}
Let $I=I_c(G)$. By Equation \eqref{eq:primary-section3}, we have
\[
\operatorname{SP}(I)=\left\{x\in\mathbb R_{\ge0}^n:\sum_{i\in e}x_i\ge1 \text{ with } e\in E(G^c),\quad \sum_{i\in T}x_i\ge1 \text{ with } T\in\mathcal T(G) \right\}.
\]
Let $v$ be a vertex of $\operatorname{SP}(I)$. First note that $v_i\le1$ for every $i$. Indeed, if $v_i>1$, then every defining inequality involving $x_i$ is strict at $v$. Hence, for sufficiently small $\varepsilon>0$, both $v+\varepsilon\mathbf e_i$ and $v-\varepsilon\mathbf e_i$ belong to $\operatorname{SP}(I)$, contradicting the fact that $v$ is a vertex.

Suppose that $|v|>n-2$, and set $d_i=1-v_i$. Then $0\le d_i\le1$ and
\begin{equation}\label{eq:defect-less-two}
\sum_{i=1}^n d_i=n-|v|<2.
\end{equation}
If an edge inequality $\sum_{i\in e}x_i\ge1$ holds with equality at $v$, then $v_i+v_j=1$ and hence $d_i+d_j=1$. Similarly, a triangle equality $\sum_{i\in T}x_i =1$ would imply $\sum_{i\in T}d_i=2$, while $v_j=0$ would give $d_j=1$. It follows from Equation \eqref{eq:defect-less-two} that no triangle inequality is an equality at $v$ and that $v_j=0$ for at most one index $j$.

Let $H$ be the spanning subgraph of $G^c$ consisting of the edges $\{i,j\}$ for which $v_i+v_j=1$. 

\medskip
\noindent\textbf{Claim.}
Every bipartite connected component of \(H\) contains an index
\(j\) such that \(v_j=0\).

\smallskip
 Indeed, let \(C=A\sqcup B\) be a bipartite connected component of \(H\), and
define \(\boldsymbol u\in\mathbb R^n\) by
\[
u_i=
\begin{cases}
1, & i\in A,\\
-1, & i\in B,\\
0, & i\notin C.
\end{cases}
\]
For every edge \(\{i,j\}\) of \(C\), we have \(u_i+u_j=0\). Hence,
\[
(v_i\pm\varepsilon u_i)+(v_j\pm\varepsilon u_j)
=v_i+v_j=1,
\]
so all edge equalities at \(v\) remain unchanged under the
perturbations \(v\pm\varepsilon\boldsymbol u\).

There is no triangle equality at \(v\), and there is at most one
index \(j_0\) such that \(v_{j_0}=0\). If no such index exists, or if
\(j_0\notin C\), then the nonnegativity equalities at \(v\) are also
preserved. Since every other defining inequality is strict
at \(v\), for sufficiently small \(\varepsilon>0\) both \(v+\varepsilon\boldsymbol u \) and \(v-\varepsilon\boldsymbol u \) belong to \(\operatorname{SP}(I)\), contradicting the fact that \(v\) is a vertex. This proves the claim.

\medskip

By the claim, if \(v\) has no zero coordinate, then \(H\) has no
bipartite connected component. If \(v_{j_0}=0\) for some index
\(j_0\), then \(j_0\) is unique by Equation \eqref{eq:defect-less-two}, and
every bipartite connected component of \(H\) must contain \(j_0\).
Since distinct connected components are disjoint, \(H\) has at most
one bipartite connected component.

If every component of $H$ is non-bipartite, then the equations $d_i+d_j=1$ force $d_i=1/2$ throughout each component. Every such component has at least three vertices, so Equation \eqref{eq:defect-less-two} forces $H$ to be connected and $n=3$. Hence, $H=K_3\subseteq G^c$, which would make every vertex of $G$ isolated, a contradiction.

Now, we consider the case in which \(H\) has a bipartite
connected component \(C\). By the claim, there is an index \(j_0\in C\)
such that \(v_{j_0}=0\). Any other component would be non-bipartite and would contribute at least $3/2$ to $\sum_i d_i$, in addition to $d_{j_0}=1$, contradicting Equation \eqref{eq:defect-less-two}. Hence, $H$ is connected and bipartite. Let $[n]=A\sqcup B$ be its bipartition with $j_0\in A$. The edge equations imply that  $d_i=1$ for $i\in A$ and $d_i=0$ for $i\in B$. Thus, $|A|<2$, so $A=\{j_0\}$. Since $H$ is connected, $j_0$ is adjacent in $G^c$ to every other vertex and is therefore isolated in $G$, again a contradiction. 

Therefore, every vertex $v$ of $\operatorname{SP}(I)$ satisfies $|v|\le n-2$, and hence $\delta(I)\le n-2$. Finally, $I_c(G)$ is generated in degree $n-2$. By Equation \eqref{eq:general-symbolic-lower}, $\reg(I^{(t)}) \ge (n-2)t$ for all $t$, and take to the limit gives $\delta(I)\ge n-2$.
\end{proof}

\begin{lemma}\label{lem:isolated-reduction}
Let $V,W\subseteq[n]$ such that $V \cap W=\emptyset$, let $J\subseteq k[x_i:i\in V]$ be a nonzero proper square-free monomial ideal, and regard $I=x_WJ$ as an ideal of $R$. Then
\[
\Min(I)=\{(x_j):j\in W\}\cup\Min(J), \text{ and } I^{(t)}=x_W^tJ^{(t)}.
\]
Consequently,
\[
\reg(I^{(t)})=t|W|+\reg(J^{(t)}) \qquad \text{ for all } t\ge1.
\]
In particular, if $W$ is the set of isolated vertices of $G$, $H=G_{[n]\setminus W}$, and $|V(H)|\ge3$, then
\begin{equation}\label{eq:isolated-reduction}
I_c(G)=x_WI_c(H),
\qquad
\reg(I_c(G)^{(t)})=t|W|+\reg(I_c(H)^{(t)}).
\end{equation}
\end{lemma}

\begin{proof}
Since the variables dividing $x_W$ do not occur in $J$,
\[
I=x_WJ=(x_W)\cap J=\left(\bigcap_{j\in W}(x_j)\right)\cap\left(\bigcap_{P\in\Min(J)}P\right),
\]
and this decomposition is irredundant. Hence, $\Min(I)=\{(x_j) :j\in W\}\cup\Min(J)$. Since the two sets of variables are disjoint, the symbolic powers satisfy
\[
I^{(t)}=\left(\bigcap_{j\in W}(x_j)^t\right)\cap J^{(t)}=(x_W^t)\cap J^{(t)}=x_W^tJ^{(t)}.
\]
Multiplication by $x_W^t$ shifts degrees by $t|W|$. Therefore, \(\reg(I^{(t)})=t|W|+\reg(J^{(t)}).\)

Finally, let $V=[n]\setminus W$. Since the vertices in $W$ are isolated, $E(G)=E(H)$ and $x_{[n]\setminus e}=x_Wx_{V(H)\setminus e}$ for $e\in E(H)$. The desired formulas now follow by applying the preceding result with $J=I_c(H)$.
\end{proof}

\begin{corollary}\label{cor:delta-and-lower}
Let $G$ be a simple graph on $[n]$, $n\ge 3$, with $E(G)\neq\emptyset$. Then   $$\delta(I_c(G))=n-2$$
and
\begin{equation}\label{eq:lower-main}
\reg(I_c(G)^{(t)})\ge(n-2)t \qquad \text{for all} \quad t \ge 1.
\end{equation}
\end{corollary}

\begin{proof}
The lower bound follows immediately from Equation \eqref{eq:general-symbolic-lower}, since $I_c(G)$ is generated in degree $n-2$. Let $W$ be the set of isolated vertices and let $H=G_{[n]\setminus W}$ and $m=|V(H)|$. If $W=\emptyset$, then the result follows from Proposition \ref{prop:delta-no-isolates}. Suppose that $W\neq\emptyset$. Since $E(G)\neq\emptyset$, necessarily $m\ge2$. If $m\ge3$, then by Proposition~\ref{prop:delta-no-isolates} and Equation \eqref{eq:isolated-reduction}, we have
\[
\delta(I_c(G))=|W|+\delta(I_c(H))=|W|+m-2=n-2,
\]
and 
\[\reg(I_c(G)^{(t)})=t|W|+\reg(I_c(H)^{(t)}) \geq t|W|+ (m-2)t = (n-2)t . \]
If $m=2$, then $H=K_2$ and $I_c(G)=(x_W)$. Therefore, we have
\[\reg(I_c(G)^{(t)})=t|W|=t(n-2) \ \text{and} \ \delta(I_c(G))=n-2 . \]
\end{proof}

\begin{proposition}\label{prop:b-value}
Let $G$ be a simple graph on $[n]$, where $n\ge3$, and suppose that $G$ has no isolated vertices. Let $I=I_c(G)$. Then
\[
b(I)=
\begin{cases}
n-2,&G\text{ is connected},\\
n-1,&G\text{ is disconnected}.
\end{cases}
\]
\end{proposition}

\begin{proof}
Let $\Delta=\Delta(I)$ and let $\Gamma$ satisfy $\mathcal F(\Gamma)\subseteq\mathcal F(\Delta)$. There exist subsets $A\subseteq E(G^c)$ and $B\subseteq\mathcal T(G)$ such that 
\[
\mathcal F(\Gamma)=\{[n]\setminus e:e\in A\}\cup\{[n]\setminus T:T\in B\}.
\]
By Alexander duality,
\[
J:=I_{\Gamma^*}
=
\left(
x_{[n]\setminus F}:F\in\mathcal F(\Gamma)
\right)
=
(x_e:e\in A)+(x_T:T\in B).
\]
The $1$-skeleton of $\Gamma^*$ has edge set \( E\bigl((\Gamma^*)^{(1)}\bigr) = \left\{\{i,j\}\subseteq [n] : \{i,j\}\notin A\right\}\), and all $n$ vertices occur. By Lemma~\ref{lem:depth-skeleton} and Terai's formula, $\reg(I_\Gamma)=\pd_R(R/J)\le n-1$. Hence, $b(I)\le n-1$.

If $G$ is connected, then $E(G)\subseteq E\bigl((\Gamma^*)^{(1)})$, and hence the $1$-skeleton of $\Gamma^*$ is connected. Thus, $\reg(I_\Gamma) \le n-2$ and $b(I)\le n-2$. On the other hand, by Equation \eqref{eq:general-symbolic-bound} with $t=1$ and Equation \eqref{eq:lower-main}, we obtain  $n-2\le\reg (I)\le b(I)$. Hence, $b(I)=n-2$.
\noindent

Suppose that \(G\) is disconnected. Let \(C\) be a connected component
of \(G\), and let
\[
X=V(C),\qquad Y=[n]\setminus X.
\]
Since \(G\) is disconnected, both \(X\) and \(Y\) are non-empty. Set
\[
A=\{\{x,y\}:x\in X,\ y\in Y\},
\qquad B=\emptyset.
\]
Since no edge of \(G\) joins \(X\) and \(Y\), we have $A\subseteq E(G^c)$. Moreover, 
\[
\bigl\{\{i,j\}\subseteq [n] : \{i,j\}\notin A\bigr\}
=
\bigl\{\{i,j\}\subseteq X\bigr\}
\cup
\bigl\{\{i,j\}\subseteq Y\bigr\}.
\]
Thus, the \(1\)-skeleton of the corresponding complex
\(\Gamma^*\) is \(K_X\sqcup K_Y\), and hence is disconnected. By Lemma~\ref{lem:depth-skeleton} and Terai's formula, $\reg(I_\Gamma)=n-1$ for the corresponding $\Gamma$. Hence, $b(I)=n-1$.
\end{proof}

\begin{corollary}\label{cor:two-values}
For every simple graph $G$ on $[n]$, $n\ge3$, with $E(G)\ne\emptyset$,
\[
(n-2)t \le \reg(\cI(G)^{(t)}) \le (n-2)t+1
\qquad  \text{ for all } t\ge1.
\]
If $c(G)=1$, then
\[
\reg(I_c(G)^{(t)})=(n-2)t.
\]
\end{corollary}
\begin{proof}
Let $W$ be the set of isolated vertices, let $H=G_{[n]\setminus W}$, and $m=|V(H)|$. Since $E(G)\neq\emptyset$, it implies $m\ge2$.

If $m=2$, then $H=K_2$ and the result follows from $I_c(G)=(x_W)$. Suppose that $m\ge3$. By Proposition~\ref{prop:delta-no-isolates}, Proposition~\ref{prop:b-value}, and Equation \eqref{eq:general-symbolic-bound},
\[
\reg(I_c(H)^{(t)}) \le
\begin{cases}
(m-2)t,&H\text{ is connected},\\
(m-2)t+1,&H\text{ is disconnected}.
\end{cases}
\]
The conclusion now follows from Lemma~\ref{lem:isolated-reduction} and the lower bound Equation \eqref{eq:lower-main}.
\end{proof}

Therefore, Corollary~\ref{cor:two-values}, reduces the problem to determining which of the two possible values occurs when \(c(G)\ge2\).
%%%%%%%%%%%%%%%%%%%%%%%%%%%%%%%%%%%%%%%%%%%%
\subsection{The formula for regularity of symbolic powers}

\begin{theorem}\label{thm:k2-component}
Let $G$ be a simple graph on $[n]$ with $c(G)\ge2$, and suppose that $G$ has a connected component isomorphic to $K_2$. Then
\[
\reg(I_c(G)^{(t)})=(n-2)t+1\qquad \text{for all  } t \ge 1.
\]
\end{theorem}

\begin{proof}
Assume first that $G$ has no isolated vertices. The upper bound follows from Corollary~\ref{cor:two-values}. For $t=1$, the equality is the ordinary-power formula of Ficarra--Moradi \cite[Theorem~4.1]{fm25}. For $t\ge2$, let $\{a,b\}$ be a $K_2$-component, and $V=[n]\setminus\{a,b\}$. Define $\boldsymbol\alpha\in\mathbb N^n$ by
\[
\alpha_a=\alpha_b=0,
\qquad
\alpha_v=t-1\quad \text{ for all } v\in V.
\]
\noindent
Then $|\boldsymbol{\alpha}|=(n-2)(t-1).$
Let \(\Delta=\Delta(I_c(G))\). By Lemma~\ref{MTr}, we have
\[
\mathcal F\bigl(\Delta_{\boldsymbol{\alpha}}(I_c(G)^{(t)})\bigr)
=
\left\{
F\in\mathcal F(\Delta):
\sum_{i\notin F}\alpha_i\le t-1
\right\}.
\]
Since \(\{a,b\}\) is a \(K_2\)-component of \(G\), neither \(a\) nor
\(b\) is adjacent to a vertex of \(V\). Hence, $\{a,v\},\{b,v\}\in E(G^c) \ \text{for every }v\in V.$ For these nonedges,
\[
\alpha_a+\alpha_v
=
\alpha_b+\alpha_v
=
t-1,
\]
and therefore $ \{[n]\setminus\{a,v\}, \  [n]\setminus\{b,v\} \} \subseteq  \mathcal F(\Delta_{\boldsymbol{\alpha}}(I_c(G)^{(t)})).$ On the other hand, if \(e\in E(G^c)\) and \(e\subseteq V\), then $\sum_{i\in e}\alpha_i=2(t-1)>t-1. $ 

Moreover, there is no triangle of \(G\) containing \(a\) or \(b\), and hence every
\(T\in\mathcal T(G)\) is contained in \(V\). Thus, $\sum_{i\in T}\alpha_i=3(t-1)>t-1.$ 

It follows that
\[
\mathcal F\bigl(\Delta_{\boldsymbol{\alpha}}(I_c(G)^{(t)})\bigr)
=
\left\{
[n]\setminus\{a,v\},
[n]\setminus\{b,v\}
:
v\in V
\right\}.
\]

Let $\Gamma=\Delta_{\boldsymbol\alpha}(I_c(G)^{(t)})$. Then \(I_{\Gamma^*}=(x_ax_v,x_bx_v:v\in V),\) which is the edge ideal of the complete bipartite graph $\{a,b\}\sqcup V$. Its independence complex is the disjoint union of the simplices on $\{a,b\}$ and $V$. Hence, Lemma~\ref{lem:depth-skeleton} together with Terai's formula yields $\reg(I_\Gamma)=n-1$.

By Lemma~\ref{Hochster-Reg}, there exists $\sigma\in\Gamma$ such that $\widetilde H_{n-3}(\lk_\Gamma(\sigma);k)\ne0$. Since every facet of $\Gamma$ has dimension $n-3$, it implies $\sigma=\emptyset$. Thus, $\widetilde H_{n-3}(\Gamma;k)\ne0$. By Takayama's formula,
\[
H_{\mathfrak m}^{n-2}(R/I_c(G)^{(t)})_{\boldsymbol\alpha}\ne0.
\]
Consequently,
\[
\reg(R/I_c(G)^{(t)})\ge|\boldsymbol\alpha|+n-2=(n-2)t,
\]
so $\reg(I_c(G)^{(t)})\ge(n-2)t+1$. This proves the result when $G$ has no isolated vertices.

For arbitrary $G$, let  $W$ be the set of isolated vertices and $H=G_{[n]\setminus W}$. Applying the preceding case to the graph $H$ and by Lemma~\ref{lem:isolated-reduction}, we get
\[
\reg(I_c(G)^{(t)})=t|W|+(|V(H)|-2)t+1=(n-2)t+1.
\]
\end{proof}

Now, we consider the graph having no connected component isomorphic to \(K_2\). First, we have the following proposition.

\begin{proposition} \label{prop:nonnegative-degrees}
Let $G$ be a simple graph on $[n]$ with no isolated vertex, $c(G)\ge2$, and suppose that $G$ has no connected component isomorphic to $K_2$. Let $t\ge2$. If $\boldsymbol{\alpha}\in\mathbb N^n$ and $H_{\mathfrak m}^i(R/I_c(G)^{(t)})_{\boldsymbol{\alpha}}\neq0,$ then
\[
|\boldsymbol{\alpha}|+i\le(n-2)t-1.
\]
\end{proposition}

\begin{proof}
Set $I=I_c(G)$ and $\Gamma=\Delta_{\boldsymbol\alpha}(I^{(t)})$. By Takayama's formula, $\widetilde H_{i-1}(\Gamma;k)\ne0$. By Lemma~\ref{MTr}, there exist $A\subseteq E(G^c)$ and $B\subseteq\mathcal T(G)$ such that
\[
\mathcal F(\Gamma)=\{[n]\setminus e:e\in A\}\cup\{[n]\setminus T:T\in B\}.
\]
Thus, the edge set of the \(1\)-skeleton of \(\Gamma^*\) is \(\bigl\{\{i,j\}\subseteq [n] : \{i,j\}\notin A\bigr\}.\) If the graph $G$ is connected, then by Lemma~\ref{lem:depth-skeleton} and Terai's formula, $\reg(I_\Gamma) \le n-2$. Since $\widetilde H_{i-1}(\Gamma;k)\ne0$, it follows from Lemma~\ref{Hochster-Reg} that $i\le n-3$. Moreover, by Lemma~\ref{HT} and Corollary~\ref{cor:delta-and-lower},
\[
|\boldsymbol{\alpha}|\le a_i(R/I^{(t)})\le(n-2)(t-1).
\]
Hence, \[
|\boldsymbol{\alpha}|+i\le(n-2)(t-1)+(n-3)=(n-2)t-1.
\]
\noindent
Suppose that the $1$-skeleton of $\Gamma^*$ is disconnected. There is a partition $[n]=X\sqcup Y$ such that every pair $\{x,y\}$ with $x\in X$ and $y\in Y$ belongs to $A$. Since $A\subseteq E(G^c)$, no edge of $G$ joins $X$ to $Y$; hence every connected component of $G$ is contained in one of these sets. Each of $X$ and $Y$ contains a component of $G$, and every such component has at least three vertices. Thus,
\begin{equation}\label{eq:parts-at-least-three}
|X|,|Y|\ge3.
\end{equation}
By Lemma~\ref{MTr}, \(\alpha_x+\alpha_y\le t-1,\) for every $x\in X$ and $y\in Y$. Let $M=\max\{\alpha_x:x\in X\}$. Then $\alpha_y\le t-1-M$ for every $y\in Y$. Therefore, we have
\[
|\boldsymbol\alpha| = \sum_{x\in X}\alpha_x
+
\sum_{y\in Y}\alpha_y \le |X|M+|Y|(t-1-M)\le(t-1)\max\{|X|,|Y|\}.
\]
In the disconnected case, by Lemma~\ref{lem:depth-skeleton}, Terai's formula and Lemma~\ref{Hochster-Reg}, we obtain  $i\le n-2$. Set $r=\min\{|X|,|Y|\}$. By Equation \eqref{eq:parts-at-least-three}, $r\ge3$, and hence
\[
|\boldsymbol\alpha|+i\le(t-1)(n-r)+(n-2)\le(n-2)t-1,
\]
since $(t-1)(r-2)\ge1$.
\end{proof}

\begin{lemma}\label{lem:universal-local}
Let \(H\) be a finite simple graph on \([m]\), where \(m\geq 2\), and
let \(J\subseteq k[x_1,\ldots,x_m]\) be a nonzero proper
square-free monomial ideal with
\[
\Min(J)
=
\{P_e:e\in E(H^c)\}
\cup
\{P_T:T\in\mathcal T(H)\}.
\]
Then $\delta(J)\leq m-1, b(J)\leq m-1,$
and consequently
\[
\operatorname{reg}(J^{(t)})
\leq (m-1)t \qquad \text{ for all } t\geq 1.
\]
\end{lemma}

\begin{proof}
Let \(v=(v_1,\ldots,v_m)\) be a vertex of
\(\operatorname{SP}(J)\). As in the proof of
Proposition~\ref{prop:delta-no-isolates}, one has $v_i\leq 1 \ (i=1,\ldots,m).$ Suppose that $|v|>m-1, $ and set $d_i=1-v_i$. Then $d_i\ge0$ and $\sum_i d_i<1$. Consequently, there is no equality of the form $v_i=0$ or $\sum_{i\in T}v_i=1$. Thus the only defining inequalities that may be equalities at $v$ are $v_i+v_j=1$, equivalently $d_i+d_j=1$. Let $Q$ be the spanning graph whose edges are precisely the pairs $\{i,j\}$ satisfying this equality. 

By the claim established in the proof of
Proposition~\ref{prop:delta-no-isolates}, every bipartite connected
component of \(Q\) would contain an index \(i\) with \(v_i=0\).
Since all coordinates of \(v\) are positive, \(Q\) has no bipartite
connected component.

Therefore, every connected component of \(Q\) is non-bipartite and contains an odd cycle. Along such a cycle, the equations $d_i+d_j=1$ force $d_i=1/2$ on all vertices of the cycle. It follows that $\sum_i d_i\ge3/2$, contradicting $\sum_i d_i<1$. Thus, $|v|\le m-1$, and it would imply  $\delta(J)\le m-1$.

Now, let \(\Gamma\) be a simplicial complex occurring in the
definition of \(b(J)\). The minimal generators of
\(I_{\Gamma^*}\) have degree two or three. Thus, \(I_{\Gamma^*}\) has no linear generator, and all \(m\) vertices
belong to \(\Gamma^*\). By Lemma~\ref{lem:depth-skeleton}, $\depth (k[\Gamma^*])\geq 1.$ Hence, by the Auslander--Buchsbaum and Terai formulas,
\[
\operatorname{reg}(I_\Gamma)
=
\operatorname{pd}_{k[x_1,\ldots,x_m]}
(k[\Gamma^*])
\leq m-1.
\]
It follows that $b(J)\leq m-1.$ Finally, by Equation \eqref{eq:general-symbolic-bound}, we have
\[
\operatorname{reg}(J^{(t)})
\leq
\delta(J)(t-1)+b(J)
\leq
(m-1)(t-1)+(m-1)
=
(m-1)t.
\]
\end{proof}

\begin{lemma}\label{lem:localized-ideal}
Let \(G\) be a simple graph on \([n]\) with no isolated vertices, let
\(\sigma\subsetneq[n]\), and let
$V=[n]\setminus\sigma,
G'=G[V], R'=k[x_i:i\in V].$
Set $J=I_c(G)R[x_i^{-1}:i\in\sigma]\cap R'.$
Then
\[
J
=
\left(\bigcap_{e\in E(G'^c)}P_e\right)
\cap
\left(\bigcap_{T\in\mathcal T(G')}P_T\right).
\]
If \(J\neq R'\), then this decomposition is irredundant, and
\[
\Min(J)
=
\{P_e:e\in E(G'^c)\}
\cup
\{P_T:T\in\mathcal T(G')\}.
\]
\noindent
Moreover, in this case, let \(t\ge1\) and
\(\boldsymbol\alpha\in\mathbb Z^n\) with
\(G_{\boldsymbol\alpha}=\sigma\), and let
\(\boldsymbol\alpha'\) be the restriction of
\(\boldsymbol\alpha\) to \(V\). Then
\[
\Delta_{\boldsymbol\alpha}
\bigl(I_c(G)^{(t)}\bigr)
=
\Delta_{\boldsymbol\alpha'}
\bigl(J^{(t)}\bigr).
\]
\end{lemma}

\begin{proof}
Let $\Delta=\Delta(I_c(G)).$ By Equation \eqref{eq:primary-section3},
\[
\mathcal F(\Delta)
=
\{[n]\setminus e:e\in E(G^c)\}
\cup
\{[n]\setminus T:T\in\mathcal T(G)\}.
\]
A facet \([n]\setminus e\) contains \(\sigma\) if and only if
\(e\subseteq V\), or equivalently, $e\in E(G'^c).$ Similarly, $[n]\setminus T $ contains $\sigma$ if and only if $T \subseteq V$, which is equivalent to $ T\in\mathcal T(G').$

If \(\sigma\notin\Delta\), then no facet of \(\Delta\) contains
\(\sigma\). Hence, $E(G'^c)=\emptyset \ \text{and} \ \mathcal T(G')=\emptyset. $ By Lemma~\ref{lem:localization}, $J=R', $ and the above decomposition follows from the convention on empty
intersections.

Suppose that \(\sigma\in\Delta\). By
Lemma~\ref{lem:localization}, $ J=I_{\lk_\Delta(\sigma)}.$ Moreover,
\[
\mathcal F(\lk_\Delta(\sigma))
=
\{V\setminus e:e\in E(G'^c)\}
\cup
\{V\setminus T:T\in\mathcal T(G')\}.
\]
Therefore,
\[
J
=
\left(\bigcap_{e\in E(G'^c)}P_e\right)
\cap
\left(\bigcap_{T\in\mathcal T(G')}P_T\right).
\]
\noindent

Since no nonedge of \(G'\) is contained in a triangle of \(G'\),  there is no
prime \(P_e\) with \(e\in E(G'^c)\), is contained in a prime \(P_T\), with
\(T\in\mathcal T(G')\), and the above decomposition is
irredundant. Consequently,
\[
\Min(J)
=
\{P_e:e\in E(G'^c)\}
\cup
\{P_T:T\in\mathcal T(G')\}.
\]

For the last assertion, the assumption \(J\neq R'\), together with Lemma~\ref{lem:localization}, implies that \(\sigma\in\Delta\). Since \(G_{\boldsymbol\alpha}=\sigma\), we have \(\boldsymbol\alpha'\in\mathbb N^V\). By Lemma~\ref{HoaTrLem},
\[
\mathcal F\left(
\Delta_{\boldsymbol\alpha}(I_c(G)^{(t)})
\right)
=
\left\{
F\in
\mathcal F\bigl(\lk_\Delta(\sigma)\bigr):
\sum_{i\in V\setminus F}\alpha_i\le t-1
\right\}.
\]
On the other hand, since $J=I_{\lk_\Delta(\sigma)} $ and \(\boldsymbol\alpha'\) is nonnegative then by Lemma~\ref{MTr}, we have
\[
\mathcal F\left(
\Delta_{\boldsymbol\alpha'}(J^{(t)})
\right)
=
\left\{
F\in
\mathcal F\bigl(\lk_\Delta(\sigma)\bigr):
\sum_{i\in V\setminus F}\alpha_i\le t-1
\right\}.
\]
Thus, the two degree complexes have the same facets, and hence
\[
\Delta_{\boldsymbol\alpha}
\bigl(I_c(G)^{(t)}\bigr)
=
\Delta_{\boldsymbol\alpha'}
\bigl(J^{(t)}\bigr).
\]
\end{proof}

\begin{theorem}\label{thm:no-k2-component}
Let $G$ be a graph on $[n]$ with $c(G)\ge2$, and suppose that no connected component of $G$ is isomorphic to $K_2$. Then
\[
\operatorname{reg}(I_c(G)^{(t)})=
\begin{cases}
n-1, & t=1,\\
(n-2)t, & t\ge2.
\end{cases}
\]
\end{theorem}

\begin{proof}
Suppose that \(G\) has no isolated vertices, and set \(I=I_c(G)\) and \(\Delta=\Delta(I)\). The case $t=1$ follows from \cite[Theorem~4.1]{fm25}. Let $t\ge2$. By Equation \eqref{eq:lower-main}, it suffices to prove $\reg(R/I^{(t)})\le(n-2)t-1$.

Let $\boldsymbol\alpha\in\mathbb Z^n$ satisfy $H_{\mathfrak m}^i(R/I^{(t)})_{\boldsymbol\alpha}\ne0.$ If $\boldsymbol\alpha\in\mathbb N^n$, then by Proposition~\ref{prop:nonnegative-degrees}, $|\boldsymbol\alpha|+i\le(n-2)t-1$. Suppose that $\boldsymbol\alpha\notin\mathbb N^n$. Let $\sigma=G_{\boldsymbol\alpha}$, and $w=|\sigma|$.  Since \(H^i_{\mathfrak m}(R/I^{(t)})_\alpha\neq0,\) by Takayama's formula we imply that $\sigma\in\Delta$. Since every facet of $\Delta$ has cardinality at most $n-2$, we have $m:=n-w\ge2$.

%We have $\sigma \neq [n]$. Indeed, if $\sigma = [n]$, then $I^{(t)} R[x_1^{-1}, \ldots, x_n^{-1}] = R[[x_1^{-1}, \ldots, x_n^{-1}]$ since $I^{(t)} \neq 0$. Thus $\Delta_\alpha(I^{(t)})$ is the void complex, contradicting the assumption that $H_{\mathfrak m}^i(R/I^{(t)})_{\boldsymbol\alpha}\ne0$. 

Set $V=[n]\setminus\sigma$, $R'=k[x_i:i\in V]$, and \(J=IR[x_j^{-1}:j\in\sigma]\cap R' \). By Lemma \ref{lem:localization}, \(J=I_{\operatorname{lk}_\Delta(\sigma)} \), so $J$ is proper. Moreover, $J\neq0$, since otherwise $\operatorname{lk}_\Delta(\sigma)$ would be the full simplex on $V$, which would imply $[n]\in\Delta$, a contradiction.

Let $\boldsymbol{\alpha}'$ be the restriction of $\boldsymbol{\alpha}$ to $V$. Since \(G_{\boldsymbol{\alpha}}=\sigma\), combining Takayama's formula and
Lemma~\ref{lem:localized-ideal} we get
\[
0\neq
\widetilde H_{i-w-1}
\left(
\Delta_{\boldsymbol{\alpha}}(I^{(t)});k
\right)
=
\widetilde H_{i-w-1}
\left(
\Delta_{\boldsymbol{\alpha}'}(J^{(t)});k
\right).
\]
Since \(\boldsymbol{\alpha}'\in\mathbb N^V\), by Takayama's
formula over \(R'\), this implies that $H^{i-w}_{\mathfrak m'}
\left(R'/J^{(t)}\right)_{\boldsymbol{\alpha}'}
\neq 0.$
%In particular, \(J\) is nonzero and proper; otherwise its degree complex would be void or a simplex and would have no nonzero reduced homology. Thus \(m\geq2\), and by Lemma~\ref{lem:localized-ideal} we have \(J=R'\) if \(m\leq1\).
Therefore, by Lemma~\ref{lem:universal-local}, $\reg(J^{(t)})\le(m-1)t$. Since every coordinate of $\boldsymbol\alpha$ indexed by $\sigma$ is at most $-1$,
\[
|\boldsymbol\alpha|+i\le|\boldsymbol\alpha'|+(i-w)\le\reg(R'/J^{(t)})=\reg(J^{(t)})-1.
\]
It follows that
\[
|\boldsymbol\alpha|+i\le(m-1)t-1=(n-w-1)t-1\le(n-2)t-1.
\]
Thus, $\reg(I^{(t)})=(n-2)t$ when $G$ has no isolated vertices.

For arbitrary $G$, let $W$ be its set of isolated vertices and let $H=G_{[n]\setminus W}$. Then $c(H)=c(G)$ and $H$ has no $K_2$-component. By Lemma~\ref{lem:isolated-reduction} and the preceding case,
\[
\reg(I_c(G))=|W|+|V(H)|-1=n-1
\]
and, for $t\ge2$,
\[
\reg(I_c(G)^{(t)})=t|W|+(|V(H)|-2)t=(n-2)t.
\]
\end{proof}

\begin{theorem} \label{thm:main-symbolic-regularity}
Let \(G\) be a graph on \([n]\), where \(n\ge3\) and
\(E(G)\ne\emptyset\). Then
\[
\reg(I_c(G))
=
\begin{cases}
n-2, & c(G)=1,\\
n-1, & c(G)\ge2.
\end{cases}
\]
For every \(t\ge2\),
\[
\reg(I_c(G)^{(t)})
=
\begin{cases}
(n-2)t+1,
& c(G)\ge2 \text{ and \(G\) has a \(K_2\)-component},\\
(n-2)t,
& \text{otherwise}.
\end{cases}
\]
\end{theorem}

\begin{proof}
The result follows immediately from Corollary~\ref{cor:two-values} and Theorems~\ref{thm:k2-component} and \ref{thm:no-k2-component}.
\end{proof}

The following formula for ordinary powers was proved by Ficarra and Moradi \cite[Theorem~4.1]{fm25}:
\begin{equation}\label{eq:ordinary-powers}
\reg(\cI(G)^t)=
\begin{cases}
(n-1)t, & 1\le t\le c(G)-2,\\
(n-2)t+c(G)-1, & t\ge c(G)-1.
\end{cases}
\end{equation}

%The folllowing corollary give the complete answer of \cite[Question 2.9]{fm25}.

\begin{corollary}\label{cor:symbolic-ordinary-comparison}
For every $t\ge1$,
\[
\reg(I_c(G)^{(t)})\le\reg(I_c(G)^t).
\]
Moreover, equality holds for every $t\ge1$ if and only if either $c(G)=1$, or $c(G)=2$ and $G$ has a $K_2$-component.
\end{corollary}

\begin{proof}
The inequality follows from Theorem~\ref{thm:main-symbolic-regularity} and Equation \eqref{eq:ordinary-powers}. If $c(G)=1$, then both sides equal $(n-2)t$. If $c(G)=2$ and $G$ has a $K_2$-component, both sides equal $(n-2)t+1$. In every other case with $c(G)\ge2$, the inequality is strict for at least one $t\ge2$.
\end{proof}

\begin{example}
The distinction in Corollary~\ref{cor:symbolic-ordinary-comparison}
can already be seen when \(c(G)=2\). Let \(G=2K_3\), then \(n=6\) and
\(G\) has no \(K_2\)-component. Hence, for every \(t\ge2\),
\[
\reg(I_c(G)^{(t)})=4t,
\qquad
\reg(I_c(G)^t)=4t+1.
\]
Thus, the inequality in Corollary~\ref{cor:symbolic-ordinary-comparison} is strict. On the other hand, if \(G=K_2\sqcup K_3\), then \(n=5\),
\(c(G)=2\), and \(G\) has a \(K_2\)-component. Therefore,
\[
\reg(I_c(G)^{(t)})
=
\reg(I_c(G)^t)
=
3t+1
\]
for every \(t\ge1\).
\end{example}

\section{Serre's condition and Cohen-Macaulayness of symbolic powers}\label{sec:depth-properties}
In this section, we provide explicit combinatorial descriptions for Serre's condition on \(R/I_c(G)\) and for the Cohen-Macaulayness of \(R/I_c(G)^{(t)}\) for all \(t\geq 1\). First, we have the following lemma.
\begin{lemma}\label{lem:isolated-reduction1}
Let $G$ be a finite simple graph on $[n]$, where $n\geq3$ and $E(G)\neq\emptyset$. Let $W$ be the set of isolated vertices of $G$, set $H=G_{[n]\setminus W} $, and assume that $W\neq\emptyset$.
\begin{enumerate}
\item If $|V(H)|=2$, then $H\cong K_2$, and $R/I_c(G)^{(t)}$ is Cohen-Macaulay for every $t\geq1$.
\item If $|V(H)|\geq3$, then $R/I_c(G)$ does not satisfy Serre's condition $(S_2)$. Moreover, $R/I_c(G)^{(t)}$ does not satisfy $(S_2)$ for any $t\geq3$.
\end{enumerate}
\end{lemma}

\begin{proof}
$(1)$ Suppose that $|V(H)|=2$. Since $H$ has no isolated vertices, we have $H\cong K_2$, and hence
\[
G\cong K_2\sqcup |W|K_1.
\]
In this case, $I_c(G)=(x_W)$ and $I_c(G)^{(t)}=(x_W^t)$ for every $t\geq1$. Therefore, $R/I_c(G)^{(t)}$ is Cohen-Macaulay.\\
\noindent
$(2)$ Assume that $|V(H)|\geq3$ and set $J=I_c(H)$. Then by Lemma~\ref{lem:isolated-reduction}, we have
\[
I_c(G)=x_WJ
\quad\text{and}\quad
\operatorname{Min}(I_c(G))
=
\{(x_j):j\in W\}\cup\operatorname{Min}(J).
\]
The prime ideal $(x_j)$ has height one, whereas every prime ideal in $\operatorname{Min}(J)$ has height two or three by Lemma \ref{lm:primary-decomposition}. Hence, $\Delta(I_c(G))$ is not pure. Therefore, by Lemma~ \ref{lem:S2-links}, we obtain that $R/I_c(G)$ does not satisfy $(S_2)$.

Moreover, for each $j \in W$, the minimal prime $(x_j)$ corresponds to a facet of $\Delta(I_c(G))$ of cadinality $n - 1$. Thus, $\dim(\Delta(I_c(G))) = n - 2 \geq 2$. If $R/I_c(G)^{(t)}$ satisfied $(S_2)$ for some $t\geq3$, then by \cite[Theorem~3.6]{Trung-Terai} we get that $\Delta(I_c(G))$ is a matroid complex. This is impossible because every matroid complex is pure.
\end{proof}

\subsection{Serre's condition for complementary edge ideals}\label{sec:serre-s2}

The following theorem answers \cite[Question~2.9]{fm25}.

\begin{theorem}\label{thm:S2}
Let $G$ be a finite simple graph on $[n]$, where $n\geq3$ and $E(G)\neq\emptyset$. Then $R/I_c(G)$ satisfies Serre's condition $(S_2)$ if and only if one of the following conditions holds:
\begin{enumerate}
\item $G$ has no isolated vertices and either $G$ is complete or $\operatorname{girth}(G)\geq5$;
\item $G\cong K_2\sqcup sK_1$ for some $s\geq1$.
\end{enumerate}
\end{theorem}

\begin{proof}
If $G$ has isolated vertices, then the assertion follows from Lemma~\ref{lem:isolated-reduction1}. Hence, it remains to consider the case in which $G$ has no isolated vertices.

Let \(\Delta=\Delta(I_c(G))\). If \(G\) is complete, then \(R/I_c(G)\) is Cohen-Macaulay by \cite[Theorem~2.8]{fm25}, and hence it satisfies \((S_2)\). Thus, it may be assumed for the rest of the proof that \(G\) is not complete.

Suppose first that \(R/I_c(G)\) satisfies \((S_2)\). Then by Lemma~\ref{lem:S2-links}, \(\Delta\) is pure, or equivalently, \(I_c(G)\) is unmixed. Since \(G\) is not complete, \(E(G^c)\ne\emptyset\), and hence, by Lemma~\ref{lm:primary-decomposition}, \(I_c(G)\) has a minimal prime of height \(2\). Also, by Lemma \ref{lm:primary-decomposition}, there is a minimal prime of height \(3\) for every triangle of \(G\). Thus, \(G\) is triangle-free.

The graph \(G\) cannot contain a \(4\)-cycle. Otherwise, let \(W\subseteq[n]\) be the vertex set of a \(4\)-cycle and \(F=[n]\setminus W\). Since \(G\) is triangle-free, the induced graph \(G_W\) is a chordless \(4\)-cycle. Its two diagonals, say \(e_1\) and \(e_2\), are therefore the only edges of \(G_W^c\). 
A facet \([n]\setminus e\) of \(\Delta\) contains \(F\) if and only if
\(e\subseteq W\). Hence the only facets of \(\Delta\) containing \(F\)
are \([n]\setminus e_1\) and \([n]\setminus e_2\). Therefore,
\[
\mathcal F\bigl(\lk_\Delta(F)\bigr)
=
\{W\setminus e_1,W\setminus e_2\}
=
\{e_2,e_1\}.
\]
Consequently, \(\lk_\Delta(F)\) is one-dimensional and disconnected, contradicting Lemma~\ref{lem:S2-links}. Hence, \(G\) contains neither a \(3\)-cycle nor a \(4\)-cycle, and therefore \(\operatorname{girth}(G)\ge5\).

Conversely, suppose that \(\operatorname{girth}(G)\ge5\). In particular, \(G\) is triangle-free, so it follows from Lemma~\ref{lm:primary-decomposition} that
\[
I_c(G)=\bigcap_{e\in E(G^c)}P_e.
\]
Thus \(\Delta\) is pure and
\[
\mathcal F(\Delta)=\{[n]\setminus e:e\in E(G^c)\}.
\]
It remains to verify the link condition in Lemma~\ref{lem:S2-links}. Let \(F\in\Delta\), set \(W=[n]\setminus F\), and assume that \(\dim(\lk_\Delta(F))\ge1\). Since \(F\in\Delta\), the graph \(G_W^c\) has an edge, and
\begin{equation}\label{eq:facets-S2-link}
\mathcal F(\lk_\Delta(F))=\{W\setminus e:e\in E(G_W^c)\}.
\end{equation}
All these facets have cardinality \(|W|-2\); hence \(\dim\lk_\Delta(F)=|W|-3\), and so \(|W|\ge4\).

If \(|W|\ge5\), then any two facets in Equation \eqref{eq:facets-S2-link} intersect, because
\[
(W\setminus e)\cap(W\setminus e')=W\setminus(e\cup e')\ne\emptyset.
\]
Therefore, $\lk_\Delta(F)$  is connected. Suppose that \(|W|=4\). In this case the facets in Equation \eqref{eq:facets-S2-link} are edges. If $\lk_\Delta(F)$ were disconnected, it would contain two disjoint facets, say \(\{a,b\}\) and \(\{c,d\}\), and no facet joining these two edges. By Equation \eqref{eq:facets-S2-link}, \(\{a,b\}\) and \(\{c,d\}\) are nonedges of \(G_W\), whereas every pair with one vertex in \(\{a,b\}\) and the other in \(\{c,d\}\) is an edge of \(G_W\). Hence \(G_W\cong K_{2,2}\), a \(4\)-cycle, contradicting \(\operatorname{girth}(G)\ge5\). Thus, every link of dimension at least one is connected. By Lemma~\ref{lem:S2-links}, \(R/I_c(G)\) satisfies \((S_2)\).
\end{proof}

\subsection{Cohen-Macaulayness of symbolic powers}\label{sec:symbolic-CM}

The preceding subsection concerns Serre's condition for \(R/I_c(G)\).
The symbolic powers $I_c(G)^{(t)}$ with $t \geq 3$ exhibit a more rigid structure. By \cite[Theorem~3.6]{Trung-Terai}, if
\(R/I_c(G)^{(t)}\) satisfies Serre's condition \((S_2)\) for some
\(t\geq3\), then the Stanley-Reisner complex
\(\Delta(I_c(G))\) is a matroid complex. This reduces the problem to
determining the graphs \(G\) for which \(\Delta(I_c(G))\) is a
matroid complex.

%Throughout this subsection, let \(G\) be a finite simple graph on \([n]\), where \(n\geq3\), and assume that \(G\) has no isolated vertices. 

We first consider two lemmas needed to determine when
\(\Delta(I_c(G))\) is a matroid complex.

\begin{lemma}\label{lem:complete-matroid}
If \(G=K_n\), then $\Delta(I_c(G))=U_{n-3,n}.$ In particular, \(\Delta(I_c(G))\) is a matroid complex.
\end{lemma}

\begin{proof}
Since $I_c(K_n)=\left(x_F:F\subseteq[n],\ |F|=n-2\right),$
the minimal nonfaces of \(\Delta(I_c(K_n))\) are precisely the \((n-2)\)-subsets of \([n]\). Therefore,
\[
\Delta(I_c(K_n))=\{F\subseteq[n]:|F|\le n-3\}=U_{n-3,n}.
\]
\end{proof}

\begin{lemma}\label{lem:forest-matroid}
Let \(G\) be a forest on \([n]\), where \(n\ge3\), and suppose that
\(G\) has no isolated vertices. Then \(\Delta(I_c(G))\) is a matroid complex if and only if either
\[
G\cong K_{1,r}\quad\text{for some }r\ge2,
\]
or
\[
G\cong qK_2\quad\text{for some }q\ge2.
\]
\end{lemma}

\begin{proof}
Let \(\Delta=\Delta(I_c(G))\). Since \(G\) is triangle-free, it follows from Lemma~\ref{lm:primary-decomposition} that
\[
I_c(G)=\bigcap_{e\in E(G^c)}P_e
\qquad\text{and}\qquad
\mathcal F(\Delta)=\{[n]\setminus e:e\in E(G^c)\}.
\]
Suppose that \(\Delta\) is a matroid complex. Since the facets of a matroid complex are its bases, then $\mathcal B(\Delta) =\{[n]\setminus e:e\in E(G^c)\},$ and the dual matroid,
\[
\mathcal B(\Delta^\perp)=E(G^c).
\]
Thus, \(E(G^c)\) is the set of bases of a rank-two matroid. The nonloop vertices of such a matroid are partitioned into parallel classes, and two vertices form a basis precisely when they lie in different classes (see~\cite[Section 1.1 ]{Oxley}). Hence, two nonloop vertices are adjacent in \(G^c\) precisely when they belong to different parallel classes. It follows that \(G^c\) is a complete multipartite graph on its nonisolated vertices, together with possible isolated vertices corresponding to loops.

If \(G^c\) has an isolated vertex \(v\), then \(v\) is adjacent in \(G\) to every other vertex. Since \(G\) is a forest, no two vertices in \([n]\setminus\{v\}\) can be adjacent; otherwise, together with \(v\), they would form a triangle. Hence, \(G\cong K_{1,n-1}\), where \(n-1\geq2\). If \(G^c\) has no isolated vertices, then \(G\), being the complement of a complete multipartite graph, is a disjoint union of complete graphs.
Since \(G\) is a forest, each of these complete graphs has at most two
vertices. As \(G\) has no isolated vertices, each has exactly two
vertices. Thus, \(G\cong qK_2\).

Conversely, suppose that \(G\cong K_{1,r}\). Then \(G^c\cong K_r\sqcup\{v\}\), and its edges are the bases of the rank-two uniform matroid on the \(r\) nonloop vertices, with \(v\) as a loop. If \(G\cong qK_2\), then \(G^c\) is the complete \(q\)-partite graph whose parts have cardinality two, so its edges are the bases of a rank-two partition matroid. Thus, in either case \(\Delta^\perp\) is a matroid complex, and so is \(\Delta\).
\end{proof}

\begin{theorem}\label{thm:symbolic-CM}
Let $G$ be a finite simple graph on $[n]$, where $n\geq 3$ and $E(G)\neq\emptyset$. The following conditions are equivalent:
\begin{enumerate}
\item $R/I_c(G)^{(t)}$ is Cohen-Macaulay for every $t\geq 1$;
\item $R/I_c(G)^{(t)}$ is Cohen-Macaulay for some $t\geq 3$;
\item $R/I_c(G)^{(t)}$ satisfies Serre's condition $(S_2)$ for some $t\geq 3$;
\item $G\cong K_n$, $G\cong K_{1,r}$ for some $r\geq 2$, $G\cong qK_2$ for some $q\geq 2$, or $G\cong K_2\sqcup sK_1$ for some $s\geq 1$.
\end{enumerate}
\end{theorem}

\begin{proof}
If $G$ has isolated vertices, then the result follows from Lemma~\ref{lem:isolated-reduction1}. Thus, we may assume that G has no isolated vertices.

Let \(\Delta=\Delta(I_c(G))\), so that \(I_c(G)=I_\Delta\). The implications (1) \(\Rightarrow\) (2) \(\Rightarrow\) (3) are immediate.

Suppose that (3) holds. If \(n=3\), then the assumption that \(G\) has no isolated vertices implies that \(G\cong K_3\) or \(G\cong K_{1,2}\), so (4) holds. Thus, it may be assumed that \(n\ge4\). If \(G\) is complete, then there is nothing to prove. Suppose that \(G\) is not complete. Then \(E(G^c)\ne\emptyset\), and by Lemma~\ref{lm:primary-decomposition}, \(\Delta\) has a facet of cardinality \(n-2\). Hence, \(\dim\Delta=n-3\ge1.\) By \cite[Theorem~3.6]{Trung-Terai}, condition (3) implies that \(\Delta\) is a matroid complex. By \cite[Theorem~3.5]{MT11}, \(R/I_\Delta\) is Cohen-Macaulay. By \cite[Theorem~2.8]{fm25}, \(G\) is therefore complete or a forest. Since \(G\) is not complete, it is a forest, and, by Lemma~\ref{lem:forest-matroid}, \(G\cong K_{1,r}\) or \(G\cong qK_2\). Thus (4) holds.

Finally, suppose that (4) holds. By Lemmas~\ref{lem:complete-matroid} and \ref{lem:forest-matroid}, \(\Delta\) is a matroid complex. By \cite[Theorem~3.5]{MT11}, \(R/I_\Delta^{(t)}=R/I_c(G)^{(t)}\) is therefore Cohen-Macaulay for every \(t\ge1\). This proves (1).
\end{proof}


\begin{thebibliography}{10}

\bibitem{fm25}
A.~Ficarra and S.~Moradi.
\newblock Complementary edge ideals, 2025.
\newblock arXiv:2508.10870.

\bibitem{FicarraMoradiMuta}
A. Ficarra, S. Moradi, and Y. Muta,
\textit{Symbolic Rees algebras of complementary edge ideals},
arXiv:2609.01165, 2026.


\bibitem{HQSM26}
T.~Hibi, A.~A. Qureshi, and S.~S. Madani.
\newblock Complementary edge ideals.
\newblock {\em Expositiones Mathematicae}, 44(4):125780, 2026.

\bibitem{HT}
T.~T. Hien and T.~N. Trung.
\newblock Regularity of symbolic powers of square-free monomial ideals.
\newblock {\em Arkiv f\"or Matematik}, 61(1):99--121, 2023.

\bibitem{HT2}
L.~T. Hoa and T.~N. Trung.
\newblock Castelnuovo--{Mumford} regularity of symbolic powers of
  two-dimensional square-free monomial ideals.
\newblock {\em Journal of Commutative Algebra}, 8(1):77--88, 2016.

\bibitem{MS}
E.~Miller and B.~Sturmfels.
\newblock {\em Combinatorial Commutative Algebra}.
\newblock Graduate Texts in Mathematics, vol.~227, Springer, New York, 2005.

\bibitem{MT}
N.~C. Minh and N.~V. Trung.
\newblock Cohen--{Macaulayness} of powers of two-dimensional squarefree
  monomial ideals.
\newblock {\em Journal of Algebra}, 322(12):4219--4227, 2009.

\bibitem{MT11}
N.~C. Minh and N.~V. Trung.
\newblock Cohen--{Macaulayness} of monomial ideals and symbolic powers of
{Stanley--Reisner} ideals.
\newblock {\em Advances in Mathematics}, 226(2):1285--1306, 2011.
\newblock Corrected version: arXiv:1003.2152v2.

\bibitem{MontanoNunez21}
Monta{\~n}o, Jonathan and N{\'u}{\~n}ez-Betancourt, Luis.
\newblock Splittings and symbolic powers of square-free monomial ideals.
\newblock {\em International Mathematics Research Notices}, (3):2304--2320, 2021.


\bibitem{Oxley}
J.~Oxley.
\newblock {\em Matroid Theory}.
\newblock Oxford University Press, Oxford, second edition, 2011.

\bibitem{PFTY}
M.R. Pournaki, S.A. Seyed Fakhari, N. Terai and S. Yassemi.
\newblock Survey article: Simplicial complexes satisfying Serre's condition: A survey with some new results.
\newblock {\em Journal of Commutative Algebra}, 6(4), 455-483, 2014.

\bibitem{RS25}
A.~Roy and K.~Saha.
\newblock Equality of ordinary and symbolic powers and the
  {Conforti--Cornu\'ejols} conjecture for $(n-2)$-uniform clutters, 2025.
\newblock arXiv:2510.15864.

\bibitem{T}
Y.~Takayama.
\newblock Combinatorial characterizations of generalized {Cohen-Macaulay}
  monomial ideals.
\newblock {\em Bulletin Math\'ematique de la Soci\'et\'e des Sciences
  Math\'ematiques de Roumanie}, 48(3):327--344, 2005.

\bibitem{T1}
N.~Terai.
\newblock Alexander duality theorem and Stanley--Reisner rings.
\newblock {\em S\={u}rikaisekikenky\={u}sho K\={o}ky\={u}roku}, 1078:174--184, 1999.
\newblock MR1715588.


\bibitem{Trung-Terai}
N.~Terai and N.~V. Trung.
\newblock Cohen--{Macaulayness} of large powers of {Stanley--Reisner} ideals.
\newblock {\em Advances in Mathematics}, 229(2):711--730, 2012.

\end{thebibliography}
	\end{document}